\documentclass[
final
, nomarks
]{dmtcs-episciences}

\usepackage[utf8]{inputenc}
\usepackage{subfigure}

\usepackage[round]{natbib}

\usepackage{graphicx}

\usepackage{enumerate}
\usepackage{mathbbol}
\usepackage{amssymb}
\usepackage{braket}
\usepackage{longtable}
\usepackage{caption}
\usepackage{mathtools}
\usepackage{microtype}
\usepackage{tikz}

\DeclareSymbolFontAlphabet{\mathbb}{AMSb}
\DeclareSymbolFontAlphabet{\mathbbl}{bbold}

\newtheorem{lemma}{Lemma}[section]
\newtheorem{theorem}[lemma]{Theorem}

\newtheorem{proposition}[lemma]{Proposition}

\newtheorem{fact}[lemma]{Fact}
\newtheorem{notation}[lemma]{Notation}
\newtheorem{remark}[lemma]{Remark}
\newtheorem{question}{Question}

\newtheorem{assumption*}{Assumption}

\newtheorem{definition}[lemma]{Definition}

\numberwithin{equation}{section}

\newcommand{\N}{\mathbb{N}}

\newcommand{\bs}{\backslash}

\newcommand\CC{{\mathcal C}}

\newcommand\GG{{\mathcal G}}
\newcommand\HH{{\mathcal H}}

\newcommand\TT{{\mathcal T}}

\newcommand\into{\hookrightarrow}

\newcommand{\tile}{\mathfrak t}

\def\Ind#1#2{#1\setbox0=\hbox{$#1x$}\kern\wd0\hbox to 0pt{\hss$#1\mid$\hss}
\lower.9\ht0\hbox to 0pt{\hss$#1\smile$\hss}\kern\wd0}

\def\notind#1#2{#1\setbox0=\hbox{$#1x$}\kern\wd0
\hbox to 0pt{\mathchardef\nn=12854\hss$#1\nn$\kern1.4\wd0\hss}
\hbox to 0pt{\hss$#1\mid$\hss}\lower.9\ht0 \hbox to 0pt{\hss$#1\smile$\hss}\kern\wd0}

\usepackage{graphics}
\usepackage[outdir=./]{epstopdf}

\author{Samuel Braunfeld\affiliationmark{1}}
\title{The undecidability of joint embedding and joint homomorphism for hereditary graph classes}

\affiliation{ University of Maryland, College Park}

\received{2019-03-29}
\revised{2019-07-23}
\accepted{2019-11-20}
\begin{document}
\publicationdetails{21}{2019}{2}{9}{5325}
\maketitle

\begin{abstract}
 We prove that the joint embedding property is undecidable for hereditary graph classes, via a reduction from the tiling problem. The proof is then adapted to show the undecidability of the joint homomorphism property as well.
\end{abstract}


\section{Introduction}

A hereditary class $\CC$ of structures has the {\em joint embedding property (JEP)} if, given $A, B \in \CC$, there exists $C \in \CC$ such that $A,B$ embed into $C$. Among other things, this is equivalent to whether $\CC$ is {\em atomic}, i.e. cannot be expressed as a union of two proper hereditary subclasses. A general strategy for understanding a hereditary class is to reduce this to understanding its atomic subclasses,  as in the following lemma for calculating growth rates in permutation classes (see \cite{Vatter} for a reference).
  
  \begin{lemma}
   Suppose $\CC$ is a permutation class, with no infinite antichain in the containment order. Then $\CC$ can be expressed as a finite union of atomic subclasses. Furthermore, the upper growth rate of $\CC$ is equal to the maximum upper growth rate among its atomic subclasses. 
   \end{lemma}
   
   Our main theorem is the following.
      
        \begin{theorem} \label{theorem:inducedJEP0}
       There is no algorithm that, given a finite set of forbidden induced subgraphs, decides whether the corresponding hereditary graph class has the JEP.
        \end{theorem}
        
        Our next result considers a variation on the JEP called the joint homomorphism property, which is of interest in infinite-domain constraint satisfaction problems \cite{Bodthes}. Modifying our proof of Theorem \ref{theorem:inducedJEP0} gives the following, answering a question of Bodirsky (personal communication).  
        
          \begin{theorem}
         There is no algorithm that, given a finite set of forbidden induced subgraphs, decides whether the corresponding hereditary graph class has the joint homomorphism property.
          \end{theorem}

Theorem \ref{theorem:inducedJEP0} is first proven for graphs enriched by a sufficient supply of unary predicates, and then a formal reduction to the pure graph language is given. A very rough sketch of the proof is as follows. The first two steps ensure that the tiling problem is equivalent to whether we can jointly embed two particular graphs, and the third step ensures that joint embedding for the class is equivalent to joint embedding for those two graphs.
        
        \begin{enumerate}
        \item Construct two graphs $A^*$, representing a grid, and $B^*$ representing a suitable collection of tiles.
        \item Choose a finite set of constraints to ensure that successfully joint embedding $A^*$ and $B^*$ encodes a solution to the tiling problem.
        \item Show that if the tiling problem admits a solution, then the chosen class admits a joint embedding procedure.
        \end{enumerate}
        
The JEP for a class $\CC$ of finite structures specified by forbidden substructures is also equivalent to $\CC$ admitting a universal object for finite structures, i.e. a countable structure avoiding the forbidden substructures and into which all members of $\CC$ embed. In graph classes, the existence of a universal object for {\em countable} structures, i.e. a countable graph avoiding the forbidden substructures and into which all other such countable graphs embed, has received much attention. In particular, in \cite{WQO} Cherlin proved the undecidability of the existence of a countable universal graph for hereditary graph classes, which serves as inspiration for the proof of Theorem \ref{theorem:inducedJEP0}.

This paper was motivated by a question of Ru\v{s}kuc on the decidability of atomicity for finitely-based permutation classes \cite{Rusk}, viewing permutations as structures in a language of two linear orders. Based on our results, we believe there is a strong possibility Ru\v{s}kuc's problem is undecidable, although it is not yet clear whether our methods are sufficient to show this. However, in forthcoming work, we intend to adapt our arguments to show the undecidability of the corresponding problem for classes of structures in a language of three linear orders.

\section{The tiling problem} \label{sec:Tiling}

Rather than using a reduction from the halting problem to prove undecidability, we will use tiling problems. The input to a tiling problem consists of a finite set $Tiles$ of tile types, as well as a set of rules of the form ``Tiles of type $i$ cannot be placed directly above tiles of type $j$'' and ``Tiles of type $k$ cannot be placed directly right of tiles of type $\ell$''. A solution to a tiling problem is a function $\tau \colon \N^2 \to Tiles$, interpreted as placing tiles on a grid, that respects the tiling rules. 

\begin{theorem}[\cite{Berger}]
There is no algorithm that, given a set of tile types and tiling rules, decides whether the corresponding tiling problem has a solution.
\end{theorem}


As we will be reducing from the tiling problem, which is co-recursively enumerable, we point out here that if $\CC$ is a hereditary class of finite structures in a finite relational language, then the JEP for $\CC$ is also co-recursively enumerable. To see this, consider $A, B \in \CC$ that can be jointly embedded, as witnessed by $C \in \CC$ and embeddings $f \colon A \to C$ and $g \colon B \to C$. As $\CC$ is hereditary, the substructure of $C$ induced on $f(A) \cup g(B)$ is also in $\CC$. Thus, given $A, B \in \CC$, there is a finite bound on the size of the possible witnesses for joint embedding, and they can be exhaustively checked.

\section{Graphs with unary predicates} \label{sec:unarypred}

In this section, we work with a language with many of the features our argument needs built in. This section's argument is then augmented with coding tricks to give arguments in the standard graph language.

\subsection{The language}

We will work in the following language.
\begin{enumerate}
\item $E$: a symmetric, irreflexive binary edge relation
\item $O^i, P'^i, G^i, T^1$ for $i \in \set{0, 1}$: unary predicates, which will denote origin vertices, non-origin path vertices, grid vertices, and tile vertices 
\item $C_i$ for $1 \leq i \leq 4$: unary predicates, which will denote additional coding vertices
\end{enumerate}

We also define a unary predicate $P^i = O^i \cup P'^i$, which will denote path vertices.


\subsection{The Canonical Models} \label{sec:canonical}
We here further flesh out steps $(1)$ and $(2)$ from the proof sketch in the introduction.

We also assume that our language contains a directed edge relation and two colored edges. We will show how to code these in our language, using the available unary predicates, in \S \ref{sec:prelim}.

Although we are concerned with the JEP for finite structures in a hereditary class $\CC$, the compactness theorem implies that the JEP for the finite members of $\CC$ is equivalent to the JEP for countable members of $\CC$. Rather than work with families of increasingly large finite structures, we prefer to take our canonical models to be countable. 

\begin{figure}[h]
\begin{center}
\includegraphics[scale=.9]{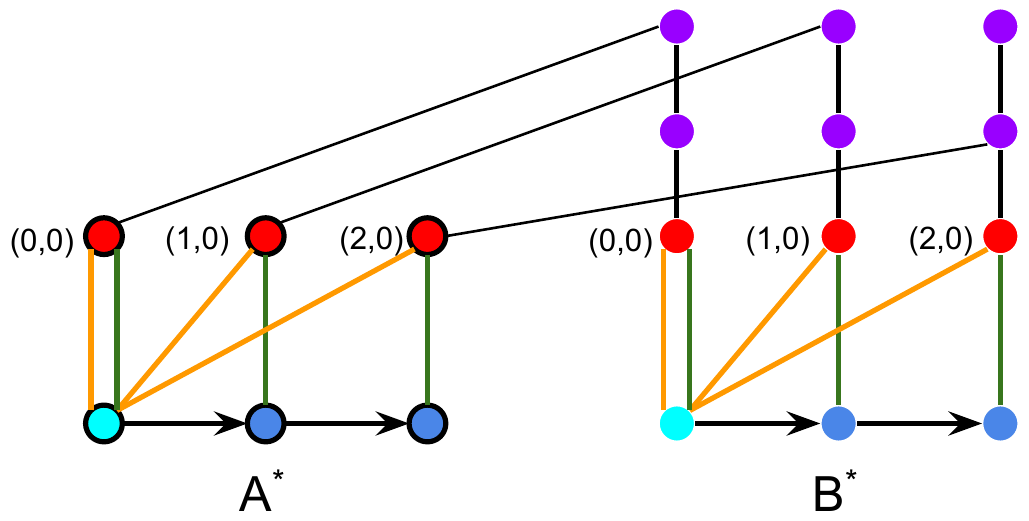}
\end{center}
\caption{A portion of the canonical models $A^*$ and $B^*$, with the grid points in $A^*$ tiled by tiles attached to grid points with the same coordinates in $B^*$. Path points are blue, with the origin a different shade. Grid points are red, their $y$-coordinate determined by an orange edge and their $x$-coordinate by a green edge. Tile points are purple. Points in 0-superscripted predicates have a black border, while points in 1-superscripted predicates do not. \\ This encodes a tiling of (0,0) with tile-type 2, (1,0) with tile-type 2, and (2,0) with tile-type 1.}
\label{fig:graph canonical}
\end{figure}

$A^*$ (see Figure \ref{fig:graph canonical}) will contain a 1-way infinite directed path, with vertices in $P^0$, and a marked origin in $O^0$. 
To every pair of points in this path, we attach a $G^0$-vertex, representing a grid point with coordinates taken from the attached path points. Because we must distinguish between $x$ and $y$-coordinates, we use the colored edges to attach each grid point to its coordinates. 

$B^*$ will look like a copy of $A^*$, using 1-superscripted predicates instead, but with a path of length $\tile$ (where $\tile$ is the number of tile types in the given tiling problem) $T^1$-vertices attached to each $G^1$-point. These represent a full tile-set available at each coordinate, with the different tile-types being distinguished by their distance from the corresponding $G^1$-point.

When we try to jointly embed $A^*$ and $B^*$, we wish our constraints to force the following: for every $G^0$-point in $A$, with coordinates $(x, y)$, we must add an edge to one tile-point attached to the $G^1$-point in $B$ with the same coordinates. This is interpreted as tiling the point $(x, y)$ by the corresponding tile-type, and our constraints should further enforce the local tiling rules.

For the particular classes of structures we are dealing with here, namely graphs with forbidden induced subgraphs, our choice of $B^*$ is rather baroque. We could have simply chosen $B^*$ to be a collection of $\tile$ tile points, with some further coding to distinguish the different tile-types. However, the construction presented here is more flexible and better adapted to handling more complex classes of structures.


\subsection{Preliminary definitions} \label{sec:prelim}
We will now precisely state our constraints, but first will establish some notation.

\begin{definition} \label{def:graphRelations}

We first define the ``special'' edges our construction uses.
\begin{enumerate}
\item $x \rightarrow^i y$ if $x, y \in P^i$ and there exist $a \in C_1, b \in C_2$ such that $xEaEbEy$. In this case, we say $x$ is the predecessor of $y$, and $y$ the successor of $x$.
\item $\Pi^i_1(v, w)$ if $v \in G^i$, $w \in P^i$, and there exists $a \in C_3$ such that $vEaEw$. In this case, we say $w$ is an $x$-projection of $v$.
\item $\Pi^i_2(v, w)$ if $v \in G^i$, $w \in P^i$, and there exists $a \in C_4$ such that $vEaEw$. In this case, we say $w$ is a $y$-projection of $v$.
\end{enumerate}

We say $g$ is a \emph{$G^i$-origin}, or sometimes a \emph{grid origin}, if there is an $x \in O^i$ such that $\Pi^i_1(g, x)$ and $\Pi^i_2(g, x)$.

Our constraints will force $x$ and $y$-projections to be unique. Given $g, g'$ in $G^i$, we say $g'$ is a horizontal successor of $g$ if they have the same $y$-projection, and the $x$-projection of $g'$ is the $\rightarrow$-successor of the $x$-projection of $g$. Similarly for vertical successor, but with $x$ and $y$ switched.

We now define binary relations related to the tiles.
\begin{enumerate}
\item For $i \in [\tile]$ (i.e. $i \in \set{1, \dots, \tile}$), we say $\tau_i(x, y)$ if $x \in G^1$, $y \in T^1$, and there exist $v_1, \dots, v_i \in T^1$ such that $v_i = y$ and $xEv_1E\dots Ev_i$. In this case, we say $y$ is a \emph{tile of type $i$ associated to $x$}.
\item $\tau(x, t)$ if $x \in G^0$, there is some $y \in G^1$ and $i \in [\tile]$ such that $\tau_i(y, t)$, and $xEt$. In this case, we say \emph{$x$ is tiled by $t$} or that \emph{$x$ is tiled by a tile of type $i$}.
\end{enumerate}

Finally, we say $x$ has a $\emph{full set of tiles}$ if there exist $t_i$ for $i \in [\tile]$ such that for all $i$, $\tau_i(x, t_i)$.
\end{definition}

\subsection{Constraints}

In addition to the constraints forcing a valid tiling to be produced when joint embedding the canonical models, we have several constraints which ensure that the origin, path, and grid points encode something grid-like. We would like to choose further constraints which ensure that every structure in our class looks like $A^*$ or $B^*$. We would like every grid point to have coordinates from the path, or every $G^1$-point to have a complete tile-set. However, as we cannot enforce such ``totality'' conditions using forbidden structures, we must allow for partial structures.

In the previous section, we noted that we would wish our constraints to force a $G^0$-point to be tiled using a tile from a $G^1$-point with the same coordinates. However, as we are forbidding a \emph{finite} number of finite structures, our constraints must have a \emph{local} character; as determining the coordinates of a grid point requires walking back to the origin, and thus looking at an unbounded number of vertices, we cannot use our constraints as desired. Instead, we will start the tiling at the origin, and then propagate it by local constraints.

Given a tiling problem $\TT$, we now define $\GG_\TT$ as the class of all finite graphs with the following constraints. Afterwards, we explicitly describe the forbidden subgraphs for some of the constraints.
\begin{enumerate}
\item \label{c:1} The unary predicates in the language are disjoint.
\item \label{c:2} A path vertex has at most 1 $\rightarrow$-predecessor.
\item \label{c:3} An origin vertex has no $\rightarrow$-predecessor.
\item \label{c:4} A grid vertex has at most 1 $x$-projection and 1 $y$-projection.
\item \label{c:5} Tile vertices are associated to at most one grid point, i.e. given $t \in T^1$, there do not exist distinct $g, h \in G^1$ such that $\tau_i(g, t)$ and $\tau_j(h, t)$.
\item \label{c:6} Tile vertices have a unique type, i.e. if $\tau_i(g, t)$ and $\tau_j(g, t)$ then $i=j$.
\item \label{c:7} The tiling rules of $\TT$ are respected.

 \item \label{c:8} If $g \in G^0$ and $h \in G^1$ are grid-origins, and $h$ has a full set of tiles, then $g$ must be tiled by a tile associated to $h$.
 
 \item \label{c:9} If a grid vertex is tiled, and there is an appropriate tile set for its neighbor, then its neighbor is also tiled. More precisely, we require the following.
 
 Suppose $g, g' \in G^0$ with $g'$ a horizontal (resp. vertical) successor of $g$, and $h, h' \in G^1$ with $h'$ a horizontal (resp. vertical) successor of $h$. Suppose $\tau(g, t)$ where $t$ is a tile vertex associated with $h$. If $h'$ has a full tileset, then $g'$ must be tiled by a tile vertex associated with $h'$.
\end{enumerate}

We note that only the last two constraints require the presence of edges, and so are the only ones that require forbidding \emph{induced} subgraphs.

We now give explicit forbidden subgraphs for some of the constraints. We sometimes forbid a non-induced subgraph; this is equivalent to forbidding the finite set of induced subgraphs obtained by adding edges in any fashion to the non-induced subgraph.
\begin{enumerate}
\item[(1)] For every pair of unary predicates, we forbid a point belonging to both predicates.
\item[(2)] For $i = 0,1$, we forbid the non-induced subgraphs consisting of points $p, p', q \in P^i$ (and the requisite coding vertices), such that $p \rightarrow^i q$ and $p' \rightarrow^i q$.
\item[(7)]Suppose $\TT$ forbids a tile of type $j$ to the right of (respectively, above) a tile of type $i$. Then we forbid the following as a non-induced subgraph.

 Let $g, g' \in G^0$ with $g'$ a horizontal (resp. vertical) successor of $g$. Let $h, h' \in G^1$ with $h'$ a horizontal (resp. vertical) successor of $h$. Finally, let $\tau(g, t_{h,i}), \tau(g', t_{h', j})$ where $t_{h, i}$ is a tile of type $i$ associated to $h$ and $t_{h', j}$ is a tile of type $j$ associated to $h'$.

\item[(8)]  Let $o_0 \in O^0, o_1 \in O^1$, $g \in G^0$, $h \in G^1$, and $t_1, \dots, t_\tile \in T^1$, with $\Pi^0_1(g,o_0)$, $\Pi^0_2(g,o_0)$, $\Pi^1_1(h,o_1)$, $\Pi^1_2(h,o_1)$, and  $hEt_1E\dots Et_\tile$. We forbid this as an induced subgraph, as well as any graph obtained by adding edges to this configuration, unless an edge is added between $g$ and some $t_i$.
\end{enumerate}

\subsection{An Informal Proof}
We wish to prove the following.

\begin{proposition} \label{prop:unaryundecidable}
Let $\TT$ be a tiling problem, and $\GG_\TT$ be the hereditary graph class defined above. Then $\GG_\TT$ has the JEP if and only if $\TT$ has a solution.
\end{proposition}

We first give an informal version of the proof, somewhat fleshing out the sketch from the introduction.

\begin{proof}
\emph{The easy direction: from the JEP to a tiling} \\
Suppose $\GG_\TT$ has the JEP. Note that $A^*, B^*$ as described above are in $\GG_\TT$, so we may jointly embed them. By constraint \ref{c:8}, the $g^0$-origin in $A^*$ must be tiled by adding an edge to a tile associated with the $g^1$-origin in $B^*$, and by constraint \ref{c:9} this must propagate to a tiling of the whole grid in $A^*$, for each $(i,j)$ adding an edge from the $g^0$-point in $A^*$ with coordinates $(i,j)$ to a tile associated with the $g^1$-point in $B^*$ with coordinates $(i,j)$. Since the tiling rules must be respected by constraint \ref{c:7}, we may then read a solution to the tiling problem off the resulting graph.

\emph{The delicate direction: from a tiling to the JEP} \\
Here, we are a bit sketchier. We first fix a solution $\theta \colon \N^2 \to [\tile]$ to the tiling problem $\TT$. Given $A, B \in \GG_\TT$, we initially take the disjoint union $C = A \sqcup B$.

As only constraints \ref{c:8} and \ref{c:9} require the presence of edges, these are the only constraints that may be violated at this point, and in fact only constraint \ref{c:8} may be. We thus use $\theta(0, 0)$ to tile all $G^0$-origins in one factor from all full tilesets attached to $G^1$-origins in the other factor. However, now there may be violations of constraint \ref{c:9} for points with coordinates $(1,0)$ and $(0,1)$. We continue using $\theta$ to appropriately tile our grids. The key point here is constraints \ref{c:2}-\ref{c:4} ensure that every grid point we must work with has well-defined coordinates, so we have a definite input to give to $\theta$.
\end{proof}

In the following two subsections, we give the formal proof of Proposition \ref{prop:unaryundecidable}.

\subsection{From the JEP to a Tiling} \label{sec:JEPtoTiling}
Suppose $\GG_\TT$ has the JEP. For this direction, we may largely repeat the informal version.

Let $\Pi^0 = \set{p^0_i | i \in \N}$, and let $\Gamma^0 = (\Pi^0)^2$, whose elements we denote $g^0_{i, j}$ rather than $(p^0_i, p^0_j)$. Let $A^*$ start with the vertex set $\Pi^0 \cup \Gamma^0$, with $p_0^0 \in O^0$, $\Pi^0\bs\set{p_0^0} \subset P'^0$, and $\Gamma^0 \subset G^0$. Also, add  coding vertices in $C_i$ and the associated edges needed to encode the relations $p_i^0 \rightarrow^0 p_{i+1}^0$ for each $p_i^0 \in \Pi^0$, and $\Pi^0_1(g_{i,j}^0, p_i^0)$ and $\Pi^0_2(g_{i,j}^0, p_j^0)$.

 Let $B^*$ be constructed as $A^*$, but using 1-superscripted points, sets, and predicates in place of 0-superscripted ones. Let $\Theta^1 = \Gamma^1 \times [\tile]$, and denote its elements as $t^1_{g, i}$ rather than $(g, i)$, Add these vertices to $B^*$, with $\Theta^1 \subset T^1$. Finally, for each $g \in \Gamma^1$, add edges so that $gEt^1_{g,1}E\dots Et^1_{g,\tile}$.

By inspection, $A^*, B^* \in \GG_\TT$. Let $C \in \GG_\TT$ jointly embed $A^*$ and $B^*$. We claim $C$ encodes a solution to $\TT$.

By constraint \ref{c:1}, no points in $A^*$ and $B^*$ got identified in $C$, except perhaps coding vertices. By constraints \ref{c:8} and \ref{c:9}, for every $(i,j) \in \N^2$ there is some $k \in [T]$ such that $\tau(g^0_{i,j}, t^1_{g_{i,j}, k})$. Define the function $\theta \colon \N^2 \to [\tile]$ by picking one such $k$ for each $(i,j)$. By constraint \ref{c:7}, $\theta$ is a solution to $\TT$.

\subsection{From a Tiling to the JEP}

For this section, we fix a solution $\theta \colon \N^2 \to [\tile]$ to $\TT$.

We begin by establishing some effects of constraints  \ref{c:2}-\ref{c:4}, which will allow us to assign coordinates to grid points. We note that, although it would add little additional overhead, it is not necessary to constrain the number of $\rightarrow$-successors, and so constraints \ref{c:2} and \ref{c:3} actually allow the path vertices to form a forest.

\begin{definition}
In any graph, let $\rightarrow^i_n$ be the $n$-fold composition of $\rightarrow^i$.

Given $p \in P^i$ and $o \in O^i$, we say \emph{$p$ is on a path with origin $o$} if there is some $n \in \N$ so that $o \rightarrow^i_n p$. In this case, we say \emph{$p$ is at distance $n$ from $o$}.

Let $G^i_*$ be the set of all $g \in G^i$ such that there exist $o \in O^i$ and $x, y \in P^i$ with $\Pi^i_1(g, x), \Pi^i_2(g, y)$ and $x$ and $y$ are on paths with origin $o$. In this case, if $x$ is at distance $n$ from $o$, and $y$ at distance $m$, we say \emph{$g$ has coordinates $(n,m)$}.
\end{definition} 

Constraints \ref{c:2} and \ref{c:3} ensure that if $p$ is on a path with origin $o$ and a path with origin $o'$, then $o= o'$. They also ensure that the distance of $p$ from $o$ is unique. This, together with constraint \ref{c:4}, ensures that the coordinates of a grid point are unique. 

\begin{definition}
Let $\theta_* \colon G^0_* \to [\tile]$ be defined by $\theta_*(g) = i$ if and only if $g$ has coordinates $(n, m)$ and $\theta(n,m) = i$.
\end{definition}

We are now ready to state our joint embedding procedure. Let $A, B \in \GG_\TT$. Let $C_0$ be the disjoint union $A \sqcup B$. We construct an extension $C$ of $C_0$ by adding edges of the form $(g, t)$ when the following conditions are met.
\begin{enumerate}
\item $(g, t) \in A \times B \cup B \times A$
\item $g \in G^0$
\item $g$ has coordinates $(n,m)$ for some $n,m \in \N$
\item There is $h \in G^1$ with coordinates $(n,m)$ such that $\tau_{\theta_*(g)}(h, t)$
\end{enumerate}

\begin{remark}
This procedure may add many more tiling-relations than would be required to satisfy the constraints. For example, we tile any grid point with coordinates, even if preceding grid points are missing that block propagation from the origin, and we may tile using tiles from incomplete tilesets.
\end{remark}

We now wish to show that $C \in \GG_\TT$ by showing it satisfies each constraint.

As constraint \ref{c:1} only involves unary predicates, and these remain unchanged by taking the disjoint union and adding edges, it remains satisfied in $C$. 

\begin{lemma}
$C$ satisfies constraints \ref{c:2}--\ref{c:6}.
\end{lemma}
\begin{proof}
For all these constraints, the forbidden configuration is connected, and thus they are satisfied in $C_0$. However, our procedure then only adds edges from $G^0$ to $T^1$-vertices, which by constraint \ref{c:1} are not of any other type. As none of the forbidden configurations involve both $G^0$ and $T^1$-vertices, such edges cannot cause them to be violated, and so they continue to be satisfied in $C$. 
\end{proof}

For the remaining constraints, the outline of the argument is the same. We consider a forbidden configuration in $C$, and show that it must have arisen from adding edges to a particular configuration in $C_0$. We then argue that our procedure would have added edges to the $C_0$-configuration so as to avoid creating the forbidden configuration in $C$.

\begin{lemma} \label{lemma:Con7}
$C$ satisfies constraint \ref{c:7}.
\end{lemma}
\begin{proof}
Again, our constraint is connected, and so satisfied in $C_0$. Fix a violation of \ref{c:7}, say of the horizontal rule, with vertices as in the constraint description. As we only add edges from $G^0$-vertices to $T^1$-vertices, we must have added either the edge $(g, t_{h,i})$ or $(g', t_{h', j})$. However, if we have only added one such edge, the configuration without that edge would be connected and would have been present in $C_0$, and so be entirely contained in one factor. This is a contradiction, as we only add edges between points in distinct factors. Thus our procedure must have added both these edges.

Thus $t_{h, i}$ is a tile of type $\theta_*(g)$ and $t_{h', j}$ is a tile of type $\theta_*(g')$, and by constraints \ref{c:5} and \ref{c:6} these types are unique. Suppose $g$ has coordinates $(n, m)$; as $g'$ is a horizontal successor of $g$, it must have coordinates $(n+1, m)$. But then $t_{h, i}$ is of type $\theta(n, m)$ and $t_{h, j}$ is of type $\theta(n+1, m)$, so they cannot violate \ref{c:7}.
\end{proof}

\begin{lemma} \label{lemma:Con8}
$C$ satisfies constraint \ref{c:8}.
\end{lemma}
\begin{proof}
Let $X=\set{g, c, d, o}, Y=\set{g', c', d', o', t'_1, \dots, t'_\tile}$, and suppose $X \cup Y$ witnesses a violation of constraint \ref{c:8}, with $o \in O^0$, $g \in G^0$ with $x$ and $y$-projections equal to $o$, and $c$ and $d$ the requisite coding vertices; let $g', c', d', o'$ be a corresponding configuration using 1-superscripted predicates, and let $t_i \in T^1$ for $1\leq i \leq \tile$ with $g'Et'_1E\dots Et'_\tile$.

As $X$ and $Y$ are each connected and neither contains both a $g^0$-point and a $T^1$-point, they must each lie in a single factor, and these factors must be distinct. Thus in $C_0$, $g$ and $g'$ both have coordinates $(0, 0)$ and $g'$ has a full tileset, so our procedure adds an edge from $g$ to $t'_{\theta(0,0)}$, and so the constraint is satisfied in $C$. 
\end{proof}

\begin{lemma}
$C$ satisfies constraint \ref{c:9}.
\end{lemma}
\begin{proof}
Consider a violation of constraint \ref{c:9}, with labels as in the constraint description (including suitable path and coding vertices). Since the constraint is connected, it must have been satisfied in $C_0$, and so the edge from $g$ to $t$ must have been added by our procedure afterward. 
As in Lemma \ref{lemma:Con8}, the violation then splits into two connected components in $C_0$, one in each factor; one component contains $g, g'$, and their associated path and coding vertices while the other contains $h, h'$, and their associated tilesets and path and coding vertices.  

As our procedure added an edge from $g$ to $t$, $g$ and $h$ must have had coordinates in $C_0$. Thus $g'$ and $h'$ also have coordinates in $C_0$. As $h'$ has a full tileset in $C_0$, our procedure adds an edge from $g'$ to a tile in this tileset, which satisfies the constraint. 
\end{proof}

\section{Moving to the language of graphs}\label{sec:movingToGraphs}
Given a finitely-constrained hereditary class $\GG_\TT$ in the language with unary predicates, we wish to produce a finitely-constrained hereditary graph class that has the JEP if and only if $\GG_\TT$ does. For this, we need some means of interpreting the unary predicates in the pure graph language. Our plan is to associate the $i^{th}$ unary predicate to some graph $G_i$, and to represent ``$v$ is in the $i^{th}$ predicate'' by freely joining a copy of $G_i$ over $v$. In order for this coding to be unambiguous, the graphs we choose must form an antichain under embeddings.

 We remark that we do not actually require an infinite antichain in the following definition, merely one with as many graphs as we have unary predicates. For our argument, the minimum size will be 13.

\begin{definition} \label{def:graphAntichain}
We now fix an infinite collection of 2-connected graphs with basepoints $(G_i, a_i)_{i \in \N}$, such that $\set{G_i}_{i \in \N}$ is an antichain under embeddability, and such that there is no automorphism of any $G_i$ moving the basepoint.
\end{definition}

\begin{definition}
Given a graph $G$, a {\em block} of $G$ is a maximal 2-connected component. Every graph has a unique decomposition into blocks.
\end{definition}

\begin{definition}
Let $\CC_k$ be the class of finite graphs with $k$ unary predicates, which we will refer to as colors $\set{1, \dots, k}$. Let $\CC^*_k \subset \CC_k$ be the subclass in which the colors partition the vertices, and in which any (colored) copy of the $G_i$ are forbidden.
\end{definition}

\begin{definition} \label{def:wedge}

Define $\wedge \colon \CC^*_k \to \set{graphs}$ as follows: for each vertex of the graph, if it has color $i$, freely attach a copy of $G_i$ over it at the basepoint, i.e. for each $v \in G$ with color $i$, take the disjoint union $G \sqcup (G_i, a_i)$, and then identify $v$ with $a_i$. These copies of $G_i$ will be called \emph{attached copies}.

The image of $A \in \CC^*_k$ will be denoted by $\widehat A$. We will also let $\widehat \GG = \set{\widehat G | G \in \GG}$.
\end{definition}

\begin{lemma} \label{lemma:attachedcopies}
Let $G \in \CC^*_k$. Any copy of $G_i$ in $\widehat G$ is an attached copy.
\end{lemma}
\begin{proof}
As $G_i$ is 2-connected, any copy must be contained in a single block of $\widehat G$. As the copies of $G_i$ are freely attached, the blocks of $\widehat G$ are those of $G$ as well as the attached $G_j$ for various $j$. Thus, any copy of $G_i$ must be contained in one of the attached $G_j$. As $\set{G_i}$ is an antichain, it must be one of the attached copies of $G_i$.
\end{proof}

\begin{definition}
Let  $\vee \colon \set{graphs} \to \CC_k$ be given by taking a graph, and for each copy of $G_i$ free over its basepoint, retaining the basepoint and giving it color $i$, and forgetting the remaining vertices.
\end{definition}

\begin{lemma} \label{lemma:hatinverse}
For any $G \in \CC_k$, $\vee(\widehat G) \cong G$. In particular, $\wedge$ is injective. 
\end{lemma}
\begin{proof}
This is immediate from Lemma \ref{lemma:attachedcopies}.
\end{proof}

\begin{lemma} \label{lemma:imageconditions}
A graph is in the image of $\wedge$ if and only if it satisfies the following properties.
\begin{enumerate}
\item For each $i$, every copy of $G_i$ is free over its basepoint.
\item If $v$ is the basepoint of a copy $H_1$ of $G_i$ and $H_2$ of $G_j$, then $H_1 = H_2$.
\item Every vertex is, for some $i$, part of a copy of $G_i$.
\end{enumerate}
\end{lemma}
\begin{remark}
(2) implicitly uses that $\set{G_i}$ is an antichain.
\end{remark}
\begin{proof}
Suppose we start with $G \in \CC^*_k$. Then $\widehat G$ is produced by making each vertex the basepoint of a copy of $G_i$, for the appropriate $i$. Thus (3) is satisfied. Conditions (1) and (2) are satisfied by Lemma \ref{lemma:attachedcopies}.

Now suppose we are given a graph $G$ of this form. By conditions (1) and (2), the vertex set of $\vee(G)$ consists of the basepoints of copies of $G_i$, each given color $i$, and with edges between them induced by $G$. Then, using condition (3), we have $G = \widehat{ \vee(G)}$.
\end{proof}

\begin{lemma}
$\wedge$ preserves embeddings, i.e. there exists an embedding $A \into B$ if and only if there exists an embedding $\widehat A \into \widehat B$
\end{lemma}
\begin{proof}
The forward direction is clear. 

For the other direction, suppose $\widehat A \into \widehat B$. Then for each copy of $G_i \subset \widehat A$, the basepoint must be mapped to such a basepoint in $\widehat B$. By Lemma \ref{lemma:imageconditions}, each of these basepoints in $\widehat B$ has a free copy of $G_i$ over it, and so can be identified with a vertex in $\vee(\widehat B)$. Furthermore, it will receive the same color as the corresponding point in $\vee(\widehat A)$. Finally, $\vee$ preserves the induced graph on the points it retains, so $\vee(\widehat A) \into \vee(\widehat B)$, and so by Lemma \ref{lemma:hatinverse} we are finished.
\end{proof}

As $\wedge$ preserves embeddings, the class $\GG_\TT$ in the language with unary predicates will have the JEP if and only if its image under $\wedge$ does. However, this image is not a hereditary graph class, and it is not clear that its downward closure will be finitely-constrained. So our goal now is to find some finitely-constrained hereditary graph class such that every member can be completed to an element in the image of $\GG_\TT$ under $\wedge$, which must satisfy the conditions of Lemma \ref{lemma:imageconditions}.

The following constraints are meant to enforce conditions (1) and (2) of Lemma \ref{lemma:imageconditions}.

\begin{definition} \label{def:HHConstraints}
Let $\HH_1$ be the set of graphs consisting of, for each $i$, a copy of $G_i$ and an additional vertex adjacent to a point that is not the basepoint of $G_i$.

Let $\HH_2$ be the set of graphs consisting of a copy of $G_i$ and $G_j$ freely joined over their basepoints, for each $i, j$, allowing $i=j$.
\end{definition}

\begin{definition}
Given a set $\GG$ of graphs, we define $\neg \GG$ to be the corresponding hereditary graph class forbidding the graphs in $\GG$.
\end{definition}

Keeping in mind condition (3) of Lemma \ref{lemma:imageconditions}, the plan for our completion algorithm is to freely attach a copy of $G_i$ for some $i$ over every vertex that is not already in some copy of one of the $\set{G_i}$. However, randomly assigning colors may produce a forbidden structure. Thus, we make sure we have a ``dummy'' color available, which is not in any non-trivial constraint, and only use its associated $G_i$ for our completion.

\begin{lemma}
Let $\GG \subset \CC_k$, such that $\neg \GG \subset \CC^*_k$. Further suppose that the only graphs in $\GG$ containing a $k$-colored vertex are multicolored single vertices and colored copies of the $\set{G_i}$.

 Then every graph in $\neg(\widehat \GG \cup \HH_1 \cup \HH_2)$ embeds into one in $\widehat {\neg \GG}$.
\end{lemma}
\begin{proof}
Let $G \in \neg (\widehat \GG \cup \HH_1 \cup \HH_2)$. Since $G \in \neg \HH_1$, it satisfies (1) from Lemma \ref{lemma:imageconditions}. Since $\GG$ contains all multicolored single vertices, then since $G \in \neg (\widehat \GG \cup \HH_2)$, it also satisfies (2) from Lemma \ref{lemma:imageconditions}.

For every vertex $v$ for which there is no $i$ such that $v$ is in copy of $G_i$ free over its basepoint, we freely attach to $v$ a copy of $G_k$, identifying $v$ with the basepoint. Call the resulting graph $G^+$, and note it satisfies (3) from Lemma \ref{lemma:imageconditions}. 

Using the 2-connectedness of the $\set{G_i}$ as in Lemma \ref{lemma:attachedcopies}, $G^+$ still satisfies (1) and (2) from Lemma \ref{lemma:imageconditions}.

We claim it is also still in $\neg \widehat \GG$, as we have only added copies of $G_k$. Suppose $\widehat H \in \widehat \GG$ embeds into $G^+$. Then $H \in \GG$ embeds into $\vee(G^+)$. 

As $G^+$ satisfies (2) from Lemma \ref{lemma:imageconditions}, $H$ cannot be a multicolored vertex. As $G^+ \in \neg \HH_1$, $H$ cannot be a colored copy of any of the $\set{G_i}$. Thus $H$ does not contain any $k$-colored vertices.

Consider the subgraph $A \subset G^+$ induced by all vertices which are not the basepoint of a freely-attached copy of $G_k$. Then $H$ must embed into $\vee(A)$. But then $\widehat H$ embeds into $A$ and thus into $G$.
\end{proof}

\begin{lemma} \label{lemma:transfer}
Let $\GG \subset \CC_k$, such that $\neg \GG \subset \CC^*_k$. Further suppose that the only graphs in $\GG$ containing a $k$-colored vertex are multicolored single vertices and colored copies of the $\set{G_i}$. Then $\neg (\widehat \GG \cup \HH_1 \cup \HH_2)$ has the JEP if and only if $\neg \GG$ has the JEP.
\end{lemma}
\begin{proof}
Suppose $\neg \GG$ has the JEP. Let $A, B \in \neg (\widehat \GG \cup \HH_1 \cup \HH_2)$. Extend them to $A^+, B^+ \in \widehat {\neg \GG}$. Then, there is some $C \in \neg \GG$ embedding $\vee(A^+), \vee(B^+)$. Thus $\widehat C$ embeds $A^+, B^+$, and so $A, B$ as well.

Now suppose $\neg (\widehat \GG \cup \HH_1 \cup \HH_2)$ has the JEP. Let $A, B \in \neg \GG$. Then there is some $C \in \neg (\widehat \GG \cup \HH_1 \cup \HH_2)$ embedding $\widehat A, \widehat B$. Extend $C$ to $C^+ \in \widehat {\neg \GG}$. Then $\vee(C^+)$ embeds $A, B$.
\end{proof}

In order to finally prove our main theorem, we must choose a suitable set $\set{(G_i, a_i)}$. The graphs must be 2-connected, form an antichain under embedding, and have no automorphism moving the basepoint. Finally, in order to have $\GG_\TT \subset \CC^*_k$, no colored version of them may embed into our canonical models $A^*, B^*$, and they must not be produced by our joint embedding process for the graphs with unary predicates. We ensure these last two points by having every edge be contained in a triangle.

\begin{notation} \label{not:wheel}
Let $W_n$ be the wheel graph on $n+1$ vertices, i.e. an $n$-cycle with an additional vertex adjacent to all others.
 
We let $G_i = W_{2i+5}$. The basepoint of $G_i$ will be the unique point of degree greater than $3$.
\end{notation}

The following fact for all wheel graphs is easy. Our restriction to $W_n$ for $n \geq 7$ and odd is only for continuity with the next section.

\begin{fact} \label{fact:wheel}
The wheel graphs $\set{W_n}$ are $2$-connected, have every edge contained in a triangle, and form an antichain under embedding. Furthermore, there is no automorphism moving our choice of basepoint.
\end{fact}

\begin{theorem} \label{theorem:inducedJEP}
  There is no algorithm that, given a finite set of forbidden induced subgraphs, decides whether the corresponding hereditary graph class has the JEP.
\end{theorem}
\begin{proof}
By Proposition \ref{prop:unaryundecidable}, it is undecidable whether $\GG_\TT$ has the JEP, as $\TT$ varies. We may modify $\GG_\TT$ to $\GG_\TT^*$ by introducing an extra color and forbidding all uncolored vertices. We also add constraints forbidding $\set{G_i}$, as well as constraints forbidding an edge between any two grid vertices. Because our joint embedding procedure only adds edges from grid vertices to tile vertices, it will not add edges between grid vertices, and so respects this constraint.

 Note that our canonical models contain no triangles (as the coding vertices break up edges), and thus no copies of the $\set{G_i}$, and they also satisfy the new constraints forbidding edges between grid vertices. Our joint embedding procedure only adds edges from grid vertices to tile vertices, and edges are only added to one tile vertex in a given set of tiles. By constraints \ref{c:5}-\ref{c:6}, there are no edges between tile vertices in distinct tile sets, so our joint embedding procedure will produce no triangles. For any $i$, every edge of $G_i$ is contained in a triangle, so no copies of $G_i$ will be produced.

 We may thus apply Lemma \ref{lemma:transfer} to $\GG^*_\TT$ to produce a family of finitely-constrained hereditary graph classes for which the JEP is undecidable as $\TT$ varies.
\end{proof}

\section{The joint homomorphism property} \label{sec:jhp}

A class of structures has the \emph{joint homomorphism property (JHP)} if, given any two structures in the class, there is a third that admits homomorphisms from both. This notion naturally arises in infinite-domain constraint satisfaction problems. For example, the constraint satisfaction problem for a theory can be realized as the constraint satisfaction problem for a particular model if and only if the models of the theory have the JHP \cite{Bodthes}. The following question was posed by Bodirsky in January 2018 (personal communication).

\begin{question}
  Is there an algorithm that, given a finite set of forbidden induced subgraphs, decides whether the corresponding hereditary graph class has the JHP?
\end{question}

In this section, our main result is a negative answer to this question, obtained by modifying our construction for the JEP.

\begin{theorem} \label{theorem:jhpUndecidable}
 There is no algorithm that, given a finite set of forbidden induced subgraphs, decides whether the corresponding hereditary graph class has the JHP.
\end{theorem}

Theorem \ref{theorem:jhpUndecidable} will be proven by modifying our proof of Theorem \ref{theorem:inducedJEP}. The reader should be familiar with the brief sketch of the proof of Theorem \ref{theorem:inducedJEP} appearing in the introduction and the discussion at the beginning of Section \ref{sec:movingToGraphs} about removing the unary predicates; relevant results and definitions will be recalled or referenced as needed.

Unlike the JEP, the JHP is sensitive to changing between quantifier-free interdefinable languages. For example, we get the following as a corollary to Theorem \ref{theorem:inducedJEP}, but will later have to work much more without the non-edge relation present.

\begin{proposition} \label{prop:nonEdgeJHP}
Work in a language with relations for edges and non-edges. Then there is no algorithm that, given a finite set of forbidden induced subgraphs, decides whether the corresponding hereditary graph class has the JHP.
\end{proposition}
\begin{proof}
Our goal is to alter the forbidden structures so that any homomorphism is actually an embedding. 

Suppose we are given a finite set $\CC^{red}$ of forbidden induced subgraphs in the language with just the edge relation. Let $\CC$ be the set of graphs, in the enriched language, with the non-edge relation added between any non-adjacent points. Let $\CC^+$ be the union of $\CC$ with the graphs on two points in which either both relations or neither relation is present, ensuring the relations act as edges and non-edges. Then we claim $\neg\CC^{red}$ has the JEP if and only if $\neg \CC^+$ has the JHP.

First, note that a homomorphism between two structures in $\neg \CC^+$ must be an embedding. Second, we may define embedding-preserving maps between $\neg\CC^{red}$ and $\neg\CC^+$ as follows. Given a graph $A \in \neg\CC^{red}$, let $A^+ \in \neg\CC^+$ be the structure obtained by adding the non-edge relation between non-adjacent vertices. Given $A \in \neg\CC^+$, let $A^{red} \in \neg \CC^{red}$ be the graph obtained by forgetting the non-edge relation. 
\end{proof}

\begin{definition}
We will say a homomorphism is {\em proper} if it is not an isomorphism.
\end{definition}

Recall that $W_n$ is a wheel graph, as in Notation \ref{not:wheel}.

\begin{fact}
Every proper homomorphic image of $W_5$ contains a copy of $K_4$.
\end{fact}

As in Proposition \ref{prop:nonEdgeJHP}, the plan for proving Theorem \ref{theorem:jhpUndecidable} will be to modify our graphs so that any $C$ witnessing the JHP also witnesses the JEP, at least in our canonical models. In Proposition \ref{prop:nonEdgeJHP}, we did this by adding the non-edge relation between any two non-adjacent vertices to make our structures clique-like. Here we do the following.
\begin{enumerate}
\item Forbid $K_4$.
\item In our canonical models, over any two non-adjacent basepoints of copies $(G_{i_1}, a_{i_1})$ and $(G_{i_2}, a_{i_2})$ (the graphs we are using to code unary predicates, see Definition \ref{def:graphAntichain}) freely join a copy of $W_5$, while keeping the vertices non-adjacent (i.e., first take the disjoint union with a copy of $W_5$, then identify two non-adjacent points of the new copy of $W_5$ with $a_{i_1}$ and $a_{i_2}$).
\end{enumerate}

The procedure above ensures that homomorphisms cannot identify the basepoints of the $(G_i, a_i)$ in our new canonical models, nor add edges between them, as this would create a copy of $K_4$. Thus the copies of $W_5$ act similarly to the non-edge relation in Proposition \ref{prop:nonEdgeJHP}. The constraint set $\HH_1$ from Definition \ref{def:HHConstraints} ensures that we cannot add an edge, nor make any identification, between a non-basepoint and any point outside the copy of $G_i$ it lies in. Thus the only possible issue is if the homomorphisms of our new canonical models fail to be embeddings within a single copy of some antichain element $G_i$.

This last possibility will be removed by forbidding all proper homomorphic images of each $G_i$ that we use from our antichain. However, these forbidden homomorphic images of $G_i$ might embed into $G_i$, or some other $G_j$, but our choice of $\set{G_i}$ will prevent this.

\begin{lemma} \label{lemma:wheel anti}
The wheel graphs, $\set{W_i}$ form an antichain under homomorphism. Also, for $i$ odd, $W_i$ is a core, i.e. any endomorphism is an automorphism.
\end{lemma}
\begin{proof}
The fact the odd wheels are cores is standard, e.g. see Example 2.23 of \cite{HT}.

Suppose $\phi$ is a homomorphism from $W_i$ to $W_j$. By Fact \ref{fact:wheel}, we already know $\phi$ cannot be an embedding, so it must either identify points or add edges. In either case, $\phi(W_i)$ will have at least two vertices of degree greater than 3. But $W_j$ has only one vertex of degree greater than 3.
\end{proof}

\begin{lemma} \label{lemma:core}
 Any homomorphic image of $W_i$ is $2$-connected.
\end{lemma}
\begin{proof}
Now let $\phi \colon W_i \to H$ be a homomorphism. Suppose $H$ becomes disconnected upon removing a vertex $v$. Then $W_i$ becomes disconnected upon removing the preimage of $v$. Let $c \in W_i$ be the vertex connected to all others. If $c \not\in \phi^{-1}(v)$, then $W_i$ is still connected after removing $\phi^{-1}(v)$. If $c \in \phi^{-1}(v)$, then no other point is in $\phi^{-1}(v)$, so again $W_i$ remains connected after its removal.
\end{proof}

\begin{notation}
As in the previous section, we let $G_i = W_{2i+5}$. The basepoint $a_i$ of $G_i$ will be the unique point of degree greater than $3$.
\end{notation}

\begin{definition}
Given a graph $G$, we construct an \emph{augmented copy of $G$}, denoted $G^+$, as follows. First, we start with a copy of $G$. Then over every non-adjacent pair of vertices, we freely join a copy of $W_5$, identifying that pair of vertices with a pair of non-adjacent vertices in $W_5$.
\end{definition}

We will use $A^+, B^+$ to denote the augmented copies of our canonical models $A^*, B^*$, where the new vertices we have added are marked with a new unary predicate $C_5$. Note that $A^+, B^+$ contain $A^*, B^*$ as induced subgraphs.

\begin{lemma} \label{lemma:unaryCanonical}
$A^+, B^+$ contain no homomorphic images, including embeddings, of any of the $\set{G_i}$, nor any copies of $K_4$.
\end{lemma}
\begin{proof}
As neither $A^*$ nor $B^*$ contain triangles, the only triangles in $A^+, B^+$ are in the copies of $W_5$ we have added. Thus there are no copies of $K_4$. Now suppose there is some $\phi \colon G_i \to H$, with $H$ embedding in $A^+$ or $B^+$. An edge $(u,v)$ of $H$ will be \emph{old} if $G_i$ contains an edge between some element of $\phi^{-1}(u)$ and some element of $\phi^{-1}(v)$, and otherwise the edge will be \emph{new}.

Let $H'$ be the graph $H$ with all new edges removed. Then every edge of $H'$ is contained in a triangle, and so must be contained in some copy of $W_5$, and so the same is true for the vertices of $H$. Since $H$ cannot be contained in a single copy of $W_5$, as the $\set{W_i}$ form an antichain under homomorphism, it must be contained in the union of multiple copies of $W_5$, say $W^1, \dots, W^n$. Since $H$ contains a vertex adjacent to all others, all the $W^i$ must intersect at a single point, and so are otherwise pairwise disjoint. But $H'$ is 2-connected, and no subgraph of $W^1 \cup \dots \cup W^n$ such that every edge connects two points in the same $W^i$ will be, unless contained within a single $W^i$.
\end{proof}

We now shift from the language with unary predicates to the pure graph language. Given the choice of $(G_i, a_i)$ to encode unary predicates, for any choice of tiling problem $\TT$ we get a hereditary graph class $\HH_\TT$, which has the JEP if and only if $\TT$ has a solution. We wish to add extra constraints to this graph class. In particular we wish to forbid $K_4$ and proper homomorphic images of the $\set{G_i}$, for $i \leq 14$. (We choose $i = 14$ because our original construction in a language with unary predicates used 12 unary predicates. We have added another predicate $C_5$ in this section, and require a ``dummy'' predicate for the translation to the pure graph language.) We will call the resulting hereditary graph class $\HH_\TT^+$.

\begin{notation}
Recall the function $\wedge$ from Definition \ref{def:wedge}. As before, we will use $\widehat{G}$ to denote $\wedge(G)$.
\end{notation}

\begin{lemma} \label{lemma:augmentedCanonical}
Let $\widehat{A^+}, \widehat{B^+}$ be the canonical models in the pure graph language, obtained by applying the function $\wedge$ to $A^+,B^+$. Then $\widehat{A^+}, \widehat{B^+}$ do not contain copies of $K_4$ or any proper homomorphic images of the $\set{G_i}$, and so are in $\HH_\TT^+$.
\end{lemma}
\begin{proof}
As $K_4$ and any homomorphic images of the $\set{G_i}$ are 2-connected, if one of them is contained in $\widehat{A^+}$ or $\widehat{B^+}$ then it must be contained in a single block. We know they are not contained in any of the copies of $\set{G_i}$ attached by $\wedge$ as the $\set{G_i}$ are cores and form an antichain under homomorphisms, so they must have been present in $A^+,B^+$. But by Lemma \ref{lemma:unaryCanonical}, we know this is not the case.
\end{proof}

As we already know $\HH_\TT$ has the JEP when $\TT$ has a solution, to check that $\HH_\TT^+$ has the JEP, it suffices to check that our joint embedding procedure for $\HH_\TT$ does not create any new copies of $K_4$ or homomorphic images of $\set{G_i}$.

Recall the two steps of our joint embedding procedure in the pure graph language. First, for every vertex $v$ such that there is no $i$ such that $v$ is in a copy of $G_i$ free over its basepoint, we attach a copy of $G_k$ freely over $v$, which gets identified with the basepoint, where $G_k$ represents a unary predicate specially reserved for this completion process (in our case, $k = 14$). We may then interpret the resulting graph in the language with unary predicates, and in the next step we add edges as we would have done there.

\begin{lemma} \label{lemma:augmentedJEP}
Let $\TT$ be a tiling problem with a solution, and suppose $A, B \in \HH_\TT^+$. Then applying our joint embedding procedure to $A,B$ creates no homomorphic images of any of the $\set{G_i}_{i \leq 14}$ except for copies of $G_{14}$, nor any copies of $K_4$, and so produces a graph in $\HH_\TT^+$.
\end{lemma}
\begin{proof}
In the first step of our joint embedding procedure, we add copies of $G_{14}$ freely over various vertices. As $K_4$ and homomorphic images of the $\set{G_i}$ are 2-connected, any new copies of these graphs must appear in the attached copies of $G_{14}$. First, $K_4$ does not embed into $G_{14}$. Then, as $G_{14}$ is a core and the $\set{G_i}$ form an antichain under homomorphisms, the only homomorphic image of any of the $\set{G_i}$ embedding in $G_{14}$ is $G_{14}$ itself. 

Let $A'$ and $B'$ be the graphs obtained from $A$ and $B$ as a result of this first step. As the graphs $\set{G_i}$ and $K_4$ are connected, no copies of $K_4$ or the $\set{G_i}$ are created by passing to the disjoint union $A' \sqcup B'$. We now continue on to the second step of our joint embedding procedure, in which edges between the factors are added to $A' \sqcup B'$. The key point in this step is that no edge we add is contained in a triangle. This immediately rules out creating any copies of $K_4$.

Now suppose our joint embedding procedure creates some graph $H$, a homomorphic image of one of the $\set{G_i}$. Let $\phi\colon G_i \to H$ be a homomorphism. We divide the edges of $H$ into two classes. An edge $(u,v)$ of $H$ will be \emph{old} if $G_i$ contains an edge between some element of $\phi^{-1}(u)$ and some element of $\phi^{-1}(v)$, and otherwise the edge will be \emph{new}.

First, note that as all the edges of $G_i$ are contained in a triangle, the same is true for all the old edges of $H$. Thus our joint embedding procedure cannot add any old edges. 

Let $H'$ be the graph $H$ with all the new edges removed. Then $H'$ must be contained in the disjoint union $A' \sqcup B'$. As $H'$ is connected, it must be contained in one of the factors. As our joint embedding procedure does not add edges within a factor, we have $H' = H$, and so $H$ was already present in one of the factors.
\end{proof}

\begin{theorem} \label{theorem:inducedJHP}
There is no algorithm that, given a finite set of forbidden induced subgraphs, decides whether the corresponding hereditary graph class has the JHP.
  
In particular, given a tiling problem $\TT$, $\HH_\TT^+$ has the JHP if and only if $\TT$ has a solution.
\end{theorem}
\begin{proof}
First, suppose $\TT$ has a solution. Then by Lemma \ref{lemma:augmentedJEP}, $\HH_\TT^+$ has the JEP, and thus the JHP.

Now suppose $\HH_\TT^+$ has the JHP. Then there is some $C \in \HH_\TT^+$ that $\widehat{A^+}, \widehat{B^+}$ both have homomorphisms into. We now wish to argue any homomorphism of $\widehat{A^+}$ into $C$ must be an embedding, and similarly for $\widehat{B^+}$.

Consider taking a homomorphism of $\widehat{A^+}$ whose image must be in $\HH_\TT^+$.  We cannot identify or add edges between any two basepoints of any of the $\set{G_i}$, as they are either already adjacent or have a copy of $W_5$ freely joined over them, so the identification or new edge would create a copy of $K_4$. We cannot identify any non-basepoint of a copy of one of the $\set{G_i}$ with any point outside of that copy of $G_i$ as that would create an edge incident to the non-basepoint, forbidden by $\HH_1$ (Definition \ref{def:HHConstraints}), unless we identified the entire copy of $G_i$ with another copy of $G_i$; however the latter is forbidden as the basepoints cannot be identified. We also cannot add an edge to a non-basepoint from outside the copy of $G_i$ it is in. Finally, we cannot add edges or identify points within a given copy of one of the $\set{G_i}$, since all proper homomorphic images of the $\set{G_i}$ are forbidden.

Thus $\widehat{A^+}, \widehat{B^+}$ actually embed in $C$, and as in Section \ref{sec:JEPtoTiling} this must encode a solution to $\TT$.
\end{proof}

\section{Questions}
Closely related to Theorem \ref{theorem:inducedJEP} is the JEP for monotone graph classes, i.e. those specified by forbidden (non-induced) subgraphs.

\begin{question}[\cite{CSS}, after Example 6] \label{q:monotone}
    Is there an algorithm that, given finite set of forbidden subgraphs, decides whether the corresponding monotone graph class has the joint embedding property?
\end{question}

When forbidding non-induced subgraphs, the constraints cannot force a joint embedding procedure to add edges between factors. Thus, when attempting to adapt the proof of Theorem \ref{theorem:inducedJEP}, the only option for tiling a grid point is through identifying points in different factors. This is a violent act that forces various uniqueness restrictions on the construction. 

In addition to the interest due to the required change in approach, Question \ref{q:monotone} arises in the program for deciding whether a monotone graph class admits a countable universal graph, as laid out in \cite{CSS}. Since the JEP is necessary for the existence of such a countable universal graph, it would seem to be a preliminary consideration. In \cite{CSS}, the additional complication of considering Question \ref{q:monotone} is intentionally avoided by assuming the forbidden subgraphs to be connected, so disjoint union serves as a joint embedding procedure. But if the answer to Question \ref{q:monotone} is positive, it would be natural to take a broader view of the decision problem for universality by allowing arbitrary forbidden subgraphs.

As mentioned earlier, the following problem of  Ru\v{s}kuc was the motivation for this paper.

    \begin{question} \label{q:Ruskuc}
    Is there an algorithm that, given finite set of forbidden permutations, decides whether the corresponding permutation class has the joint embedding property?
    \end{question}
    
The main difficulty in working with permutation classes seems to be the transitivity of the orders. When performing joint embedding, once we decide how to relate a single point in each factor to each other, many other such relations are forced by transitivity. This is in contrast to the induced subgraph case, where adding an edge between factors only forces us to add another edge if there is a conflict with one of the forbidden subgraphs, which we have complete control over choosing. 

A partial step towards Question \ref{q:Ruskuc} would be to consider the JEP for \emph{permutation graphs}. Given a permutation, this is the corresponding graph with the same vertex set and with edges defined by $xEy$ if and only if the two orders disagree between $x$ and $y$.

\begin{question} \label{q:permgraphs}
    Is there an algorithm that, given finite set of forbidden permutations, decides whether the corresponding hereditary permutation graph class has the joint embedding property?
\end{question}

In \cite{Gallai}, Gallai characterized permutation graphs in terms of an infinite family of forbidden induced subgraphs. (The characterization is more easily available in \cite{graphclasses}.) Thus, one could approach Question \ref{q:permgraphs} by attempting to modify the proof of Theorem \ref{theorem:inducedJEP} to avoid the graphs on Gallai's list. 

\acknowledgments
I thank Manuel Bodirsky for suggesting the problem addressed in Section \ref{sec:jhp}, Gregory Cherlin for many discussions on the material in this paper, and Nik Ru\v{s}kuc for helping to enormously improve the paper's presentation. I also thank the referee for many corrections and helpful comments.

The material in this paper appeared in the author's thesis \cite{BraunThesis}.

\nocite{*}
\bibliographystyle{abbrvnat}
\bibliography{jep-bib}




\end{document}